\documentclass[leqno,12pt]{amsart} 
\usepackage{amssymb}
\theoremstyle{plain} 
\newtheorem{theorem}{\indent\sc Theorem}[section]
\newtheorem{lemma}[theorem]{\indent\sc Lemma}
\newtheorem{corollary}[theorem]{\indent\sc Corollary}
\newtheorem{proposition}[theorem]{\indent\sc Proposition}

\theoremstyle{definition} 

\newtheorem{remark}[theorem]{\indent\sc Remark}

\def\C{{\mathbf{C}}}
\def\R{{\mathbf{R}}}
\def\Z{{\mathbf{Z}}}
\def\Pi{{\mathbf{P}}}
\def\Si{{\mathbf{S}}}

\begin{document}

\title[complete minimal Lagrangian surfaces]
{The Gauss map and total curvature of complete minimal Lagrangian surfaces in the complex two-space} 

\author[R.~Aiyama]{Reiko Aiyama}

\author[K.~Akutagawa]{Kazuo Akutagawa} 

\author[Y.~Kawakami]{Yu Kawakami} 

\renewcommand{\thefootnote}{\fnsymbol{footnote}}
\footnote[0]{2010\textit{ Mathematics Subject Classification}.
Primary 53C42, 53D10; Secondary 30D35, 53D12.}
\thanks{
The second author is supported in part by the Grant-in-Aid for Challenging Exploratory Research, 
Japan Society for the Promotion of Science, No.~24654009.}
\thanks{
The third author is supported by the Grant-in-Aid for Young Scientists (B), Japan Society for the 
Promotion of Science,  No. 24740044.} 
\address{
Institute of Mathematics, \endgraf
University of Tsukuba, \endgraf
Tsukuba 305-8571, Japan
}
\email{aiyama@math.tsukuba.ac.jp}

\address{
Department of Mathematics, \endgraf
Tokyo Institute of Technology, \endgraf
Tokyo 152-8551, Japan
}
\email{akutagawa@math.titech.ac.jp}

\address{
Graduate School of Natural Science and Technology, \endgraf
Kanazawa university, \endgraf
Kanazawa, 920-1192, Japan
}
\email{y-kwkami@se.kanazawa-u.ac.jp}

\maketitle

\begin{abstract}
The purpose of this paper is to reveal the relationship 
between the total curvature and the global behavior of the Gauss map of a complete minimal 
Lagrangian surface in the complex two-space. To achieve this purpose, 
we show the precise maximal number of exceptional values of the Gauss map for a complete minimal Lagrangian 
surface with finite total curvature in the complex two-space. 
Moreover, we prove that if the Gauss map of a complete minimal Lagrangian 
surface which is not a Lagrangian plane omits three values, then it takes all other values 
infinitely many times. 
\end{abstract}

\section{Introduction}
There are many similarities between surfaces in Euclidean $3$-space ${\R}^{3}$ and Lagrangian surfaces 
in the complex $2$-space ${\C}^{2}$, in particular, the case of minimal surfaces. 
In fact, there exists a representation for a minimal Lagrangian surface $M (\subset {\C}^{2})$ in terms of 
holomorphic data, similar to Weierstrass representation for a minimal surface in ${\R}^{3}$ (cf. \cite{Os1986}). 
Moreover, the Gauss map $g$ of $M$ is a holomorphic map to the unit $2$-sphere $S^{2}$. 
On the representation for $M$, Chen-Morvan \cite{CM1987} proved that there exists an explicit correspondence 
in ${\C}^{2}$ between minimal Lagrangian surfaces and holomorphic curves with a nondegenerate condition. 
Indeed, this correspondence is given by exchanging the orthogonal complex structure $J$ in ${\C}^{2}$ to another one 
on ${\R}^{4}={\C}^{2}$. For the complete case, this result can also be proved from \cite[Theorem II]{Mi1984} and 
the well-known fact \cite{HL1982} that any minimal Lagrangian submanifold in ${\C}^{n}$ is stable. 
More generally, H\'elein-Romon \cite{HR2000, HR2002} and the first author \cite{Ai2001, Ai2004} 
proved that every Lagrangian surface $S$ in ${\C}^{2}$, not necessarily minimal, is represented in terms of 
a plus spinor (or a minus spinor) of the $\text{spin}^{\C}$ bundle $(\underline{\C}_S\oplus \underline{\C}_S)\oplus (K^{-1}_{S}\oplus K_{S})$ 
satisfying the Dirac equation with potential (see \cite[Section\,1]{Ai2004} for details). 
Here, $\underline{\C}_S$ and $K_{S}$ denote respectively the trivial complex line bundle and the canonical complex line bundle of $S$. 
Remark that the representation in terms of plus spinors in $\Gamma (\underline{\C}_S\oplus \underline{\C}_S) = \Gamma (S\times {\C}^{2})$ given by 
the first author is a natural generalization of the one given by Chen-Morvan. Combining these results, we get the following: 

\begin{theorem}{$($\cite{CM1987}, \cite{Ai2001, Ai2004}$)$}\label{thm:Weierstrass}
Let $M$ be a Riemann surface with an isothermal coordinate $z=u+\mathbf{i}\, v$ around each point. Let $F = (F_{1}, F_{2})\colon M\to 
{\C}^{2}$ be a holomorphic map satisfying $|S_{1}|^{2}+|S_{2}|^{2}\not= 0$ everywhere on $M$, 
where $S_{1}:= (F_{2})'_{z} = dF_2/dz$ and $S_{2}:= - (F_{1})'_{z} = - dF_1/dz$. Then 
\begin{equation}\label{eq:Weierstrass}
f=\dfrac{1}{\sqrt{2}}e^{\mathbf{i}\,\beta/2}(F_{1}-\mathbf{i}\, \overline{F_{2}}, F_{2}+\mathbf{i}\, \overline{F_{1}}) 
\end{equation}
is a minimal Lagrangian conformal immersion from $M$ to ${\C}^{2}$ with constant Lagrangian angle $\beta \in {\R}/2\pi\Z$. 
The induced metric $ds^{2}$ on $M$ by $f$ and its Gaussian curvature $K_{ds^{2}}$ are respectively given by 
\begin{equation}\label{eq:metric-curvature}
ds^{2}=(|S_{1}|^{2}+|S_{2}|^{2})|dz|^{2}, \qquad K_{ds^{2}}=-2\dfrac{|S_{1}(S_{2})_{z}-S_{2}(S_{1})_{z}|}{(|S_{1}|^{2}+|S_{2}|^{2})^{3}}.\end{equation}
The Gauss map $g$ of $M$ is also given by 
\begin{equation}\label{eq:gauss-map}
g=[-S_{2} : S_{1}]=(-S_{2}/S_{1})\colon M\to \C P^{1}=\C\cup \{\infty \}\simeq S^{2}. 
\end{equation}
Here, by the identification of $S^{2}$ with $S^{2}(1)\times \{(e^{\mathbf{i}\beta}, 0)\}\subset {\R}^{3}\times {\R}^{3}$, 
$g$ can be regarded as the generalized Gauss map of $F (M)$ in ${\R}^{4}={\C}^{2}$ $(${\rm cf. \cite{HO1980}, \cite{HO1985}}$)$. 
Conversely, every minimal Lagrangian immersion $f\colon M\to {\C}^{2}$ with constant Lagrangian angle $\beta$ is congruent 
with the one constructed as above. 
\end{theorem}

\begin{remark}
Set a holomorphic $1$-form $hdz:=S_{1}dz$ on $M$. In terms of the Weierstrass data $(hdz, g)$ of $M$, 
the induced metric $ds^{2}$ and its Gaussian curvature $K_{ds^{2}}$ can be rewritten respectively by 
\begin{equation}\label{eq:metric-curvature2}
ds^{2}=|h|^{2}(1+|g|^{2})|dz|^{2}, \qquad K_{ds^{2}}=-\dfrac{2|g'_{z}|^{2}}{|h|^{2}(1+|g|^{2})^{3}}. 
\end{equation}
Moreover, the minimal Lagrangian immersion $f\colon M\to {\C}^{2}$ is also given by 
\begin{equation}\label{eq:Weierstrass2}
f = \dfrac{1}{\sqrt{2}}e^{\mathbf{i}\,\beta/2}\biggl{(}\int ghdz - \mathbf{i}\int \overline{hdz},\ 
\int hdz + \mathbf{i} \int \overline{ghdz} \biggr{)} \qquad \ \ 
\end{equation}
\begin{equation*} 
= \dfrac{1}{\sqrt{2}}e^{\mathbf{i}\,\beta/2}\biggl{(}- \int S_2dz - \mathbf{i}\int \bar{S_1}d\bar{z},\ 
\int \overline{\bar{S_1}d\bar{z}} + \mathbf{i} \int \overline{S_2dz} \biggr{)}
\end{equation*}
for $(\bar{S_1}d\bar{z}, S_2dz) \in \Gamma (K^{-1}_{S}\oplus K_{S})$, 
which is essentially same as the one proved by H\'elein-Romon \cite{HR2000}. 
On the other hand, the induced metric $d\tilde{s}^{2}$ and its Gaussian curvature $K_{d\tilde{s}^{2}}$ of 
a minimal surface in ${\R}^{3}$ associated with the Weierstrass data $(hdz, g)$ are given respectively by (cf. \cite{Os1986}) 
\begin{equation}\label{eq:Weierstrass-minimal}
d\tilde{s}^{2}=|h|^{2}(1+|g|^{2})^{2}|dz|^{2}, \qquad K_{d\tilde{s}^{2}}=-\dfrac{4|g'_{z}|^{2}}{|h|^{2}(1+|g|^{2})^{4}}.
\end{equation}
\end{remark}

With these understanding, 
for each $m \in {\bf N}$ and a Weierstrass data $(hdz, g)$ on an open Riemann surface $M$, 
one can consider the conformal metric $ds^2_m := |h|^{2}(1+|g|^{2})^m|dz|^{2}$ on $M$. 
In fact, the third author \cite{Ka2013} has studied the precise maximal number of exceptional values of $g$, provided that $g$ is not constant and that $ds^2_m$ is complete. 
Here, we call a value that a function or map never takes an {\it exceptional value} 
of the function or map.

The purpose of the present paper is to reveal the relationship between the total curvature and the global behavior of the Gauss map of 
a complete minimal Lagrangian surface in ${\C}^{2}$. 
Our main theorems are Theorems \ref{thm:KKM} and \ref{great-Picard1} in Section 2. 
The paper is organized as follows: In Section \ref{section2}, we first give a curvature bound for a minimal Lagrangian surface 
in ${\C}^{2}$ (Theorem \ref{thm:exceptional1}) and show that the precise maximal number of exceptional values of the Gauss map 
of a complete minimal Lagrangian surface in ${\C}^{2}$ is ``$3$''(Corollary \ref{cor:exceptional1}). 
These results follow from directly Theorem 2.1 and Proposition 2.4 in \cite{Ka2013}. 
Next, we give the upper bound for the number of exceptional values of the Gauss map of a complete minimal 
Lagrangian surface with finite total curvature in ${\C}^{2}$ (Theorem \ref{thm:KKM}). In particular, we show that 
the precise maximal number of exceptional values of the Gauss map for this class is ``$2$''. The proof is given in 
Section \ref{section3-1}. 
Moreover, by refining the Mo-Osserman argument in \cite{MO1990}, we prove that if the Gauss map of a complete minimal 
Lagrangian surface in ${\C}^{2}$ takes on ``$4$'' distinct values only a finite number of times, then the surface has finite 
total curvature (Theorem \ref{great-Picard1}). The proof is given in Section \ref{section3-2}. As a corollary of the result, 
we obtain that if the Gauss map of a complete minimal Lagrangian surface in ${\C}^{2}$ which is not a Lagrangian plane omits 
``$3$'' values, then it takes all other values infinitely many times (Corollary \ref{great-Picard2}). 
This result is a sharping of Corollary \ref{cor:exceptional1}. 

Finally, the authors would like to thank Professors Katsuei Kenmotsu, Reiko Miyaoka, Masaaki Umehara and Kotaro Yamada 
for useful advice and continuous encouragements. 

\section{Main results}\label{section2}
We first give a curvature bound for a minimal Lagrangian surface $M$ in ${\C}^{2}$. 
Set $\omega :=hdz$. Then the induced metric $ds^{2}$ on $M$ by $f$ 
is rewritten by 
\begin{equation}\label{eq:metric3}
ds^{2}=|h|^{2}(1+|g|^{2})|dz|^{2}=(1+|g|^{2})|\omega|^{2}. 
\end{equation}
Applying Theorem 2.1 in \cite{Ka2013} to the metric $ds^{2}$, we can get the following theorem. 

\begin{theorem}\label{thm:exceptional1}
Let $M$ be a minimal Lagrangian surface in ${\C}^{2}$ whose Gauss map $g\colon M\to \C\cup \{\infty \}$ omits more than three 
distinct values. Then there exists a positive constant $C$ depending on the set of exceptional values, but not $M$, such that for all 
$p\in M$ we have 
$$
|K_{ds^{2}}(p)|^{1/2}\leq \dfrac{C}{d(p)}, 
$$
where $K_{ds^{2}}(p)$ stands for the Gaussian curvature of $M$ at $p$ 
and $d(p)$ stands for the geodesic distance from $p$ to the boundary of $M$. 
\end{theorem}

Combining this with Proposition 2.4 in \cite{Ka2013}, we give the precise maximal number of exceptional values of 
the Gauss map of the case where $M$ is complete. 

\begin{corollary}\label{cor:exceptional1}
The Gauss map of a complete minimal Lagrangian surface in ${\C}^{2}$ which is not a 
Lagrangian plane can omit at most three values. Moreover, let $E$ be an arbitrary set of $q$ points 
on $\C\cup \{\infty \}$, where $q\leq 3$. Then there exists a complete minimal Lagrangian surface in 
${\C}^{2}$ whose image under the Gauss map omits precisely the set $E$. 
\end{corollary}

Note that an example of the case where $E=\emptyset$ (i.e., $q=0$) is given by the following: 
$$
F=(F_{1}, F_{2})\colon \C \to {\C}^{2}, \quad (F_{1}, F_{2})= \biggl{(}\dfrac{z^{3}}{3}+z+c_{1},\, \dfrac{z^{2}}{2}+c_{2} \biggr{)} 
\quad (c_{1}, c_{2}\in \C). 
$$
Indeed, the Weierstrass data is $(hdz, g)=(zdz, (z+1)/z)$ and the resulting surface is a complete minimal Lagrangian surface in 
${\C}^{2}$ whose Gauss map is surjective. 
See \cite{Ka2014} for more details on other function-theoretic properties 
(e.g., ramification theorem, unicity theorem) of the Gauss map for this class. 

Next, we consider the case where $M$ is complete and has finite total curvature. 
From (\ref{eq:metric-curvature2}), the Gaussian curvature $K_{ds^{2}}$ is nonpositive, and 
the total curvature is given by  
\begin{equation}\label{eq:total-curvature}
\displaystyle \int_{M}K_{ds^{2}}dA=-\dfrac{1}{2}\int_{M} \biggl{(}\dfrac{2|g'_z|}{1+|g|^{2}} \biggr{)}^{2}du\wedge dv, 
\end{equation}
where $z=u+\mathbf{i}\, v$. 

\begin{proposition}\label{pro:algebra}
Let $M$ be a complete minimal Lagrangian surface with finite total curvature in ${\C}^{2}$. Then it satisfies 
\begin{enumerate}
\item[(a)] $M$ is conformally equivalent to $\overline{M}_{\gamma}\backslash \{p_{1}, \cdots , p_{k}\}$, where $\overline{M}_{\gamma}$ 
is a compact Riemann surface of genus $\gamma$, and $p_{1}, \cdots, p_{k}\in \overline{M}_{\gamma}$ $($\cite{Hu1957}$)$, 
\item[(b)] The Weierstrass data $(hdz, g)$ is extended meromorphically to $\overline{M}_{\gamma}$ $($\cite{Os1986}$)$. 
\end{enumerate}
\end{proposition}

We call the points $\{p_{1}, \cdots, p_{k}\}$ {\it ends} of $M$. Thus we easily show that the total curvature of a complete minimal 
Lagrangian surface in ${\C}^{2}$ can only take the values $-2\pi m$, $m$ a non-negative integer, or $-\infty$. 

We give the upper bound for the number of exceptional values of the Gauss map of a complete minimal Lagrangian surface 
with finite total curvature in ${\C}^{2}$. 

\begin{theorem}\label{thm:KKM}
Let $M=\overline{M}_{\gamma}\backslash \{p_{1}, \cdots, p_{k}\}$ be a complete minimal Lagrangian surface with finite 
total curvature in ${\C}^{2}$ 
and $g\colon M\to \C\cup \{\infty \}$ the Gauss map. 
Let $d$ be the degree of $g$ considered as a map of $\overline{M}_{\gamma}$. 
Assume that $M$ is not a Lagrangian plane. 
Then we have 
\begin{equation}\label{inequ:KKM1}
D_{g}:= \sharp (({\C} \cup \{\infty \}) \backslash g(M))\leq 2+\dfrac{2}{R}, \quad \dfrac{1}{R}=\dfrac{\gamma -1+k/2}{d}<\dfrac{1}{2}. 
\end{equation}
In particular, the Gauss map can omit at most two values. 
\end{theorem}

The number two in Theorem \ref{thm:KKM} is optimal. A famous example of a complete minimal Lagrangian surface 
with finite total curvature in ${\C}^{2}$ whose Gauss map omits exactly two values is the Lagrangian catenoid (\cite{CU1999}). 
Indeed, the surface is defined as 
$$
F=(F_{1}, F_{2})\colon \C\backslash \{0\} \to {\C}^{2}, \quad (F_{1}, F_{2})=\biggl{(}z,\, \dfrac{1}{z} \biggr{)}. 
$$
The Weierstrass data is given by $(hdz, g)=(-dz/z^{2}, -z^{2})$ and the Gauss map omits exactly two values, $0$ and $\infty$. 
Note that Castro and Urbano in \cite{CU1999} proved that the only Lagrangian catenoid is a minimal Lagrangian surface 
with circle symmetry in ${\C}^{2}$. 

\begin{remark}\label{rmk:KKM}
In the case of a complete minimal surface with finite total curvature in ${\R}^{3}$, 
Osserman \cite{Os1964} proved that the Gauss map can omit at most three values. 
However, at present, there exists no example whose Gauss map omits precise three values 
(see \cite{Fa1993}, \cite{KKM2008} and \cite{WX1987} for details). 
\end{remark}

Finally, by refining the Mo-Osserman argument in \cite{MO1990}, we get the following: 

\begin{theorem}\label{great-Picard1}
Let $M$ be a complete minimal Lagrangian surface in ${\C}^{2}$. If the Gauss map takes on four values only 
a finite number of times, then $M$ has finite total curvature. 
\end{theorem}

On the other hand, by Theorem \ref{thm:KKM}, a complete minimal Lagrangian surface with finite total curvature 
whose Gauss map omits more than two values must be a Lagrangian plane. One consequence of Theorem \ref{great-Picard1} 
is therefore: 

\begin{corollary}\label{great-Picard2}
Let $M$ be a complete minimal Lagrangian surface in ${\C}^{2}$ which is not a Lagrangian plane. 
If the Gauss map of $M$ omits three values, then it takes all other values infinitely many times. 
\end{corollary}


\section{Proof of Main theorems}\label{section3}

\subsection{Proof of Theorem \ref{thm:KKM}}\label{section3-1}
Since the total curvature is finite, we may assume that $\overline{M}_{\gamma}\backslash \{p_{1}, 
\cdots , p_{k}\}$, where $\overline{M}_{\gamma}$ is a compact Riemann surface of genus 
$\gamma$, and $p_{1}, \cdots, p_{k}\in \overline{M}_{\gamma}$. 
By a rotation of the surface, we may assume that the Gauss map $g$ has neither zero 
nor pole at $p_{j}$ and that the zeros and the poles of $g$ are simple. 
The simple poles of $g$ coincides with the simple zeros of $hdz$ because the metric $ds^{2}$ is nondegenerate. 
By the completeness of $M$, $hdz$ has poles at each end $p_{j}$ (\cite{Ma1963}, \cite[Lemma 9.6]{Os1986}). 
Moreover, since the surface is well-defined on $M$, $hdz$ has poles of order ${\mu}_{j}\geq 2$ at $p_{j}$. 
Indeed, suppose that ${\mu}_{j}=1$. Then we can expand the function $h(z)$ about each end $p_{j}$ as 
$$
h(z)=\dfrac{c^{\, j}_{-1}}{z-p_{j}}+\displaystyle \sum_{n=0}^{\infty}c^{\, j}_{n}z^{n}, \quad c^{\, j}_{1}\not =0. 
$$
We also are able to expand the function $g(z)h(z)$ about each end $p_{j}$ as 
$$
g(z)h(z)=\dfrac{d^{\, j}_{-1}}{z-p_{j}}+\displaystyle \sum_{n=0}^{\infty}d^{\, j}_{n}z^{n}
$$
because $g$ is holomorphic around $p_{j}$. By (\ref{eq:Weierstrass2}), if the surface is well-defined on $M$, 
then we have 
$$
\displaystyle \int_{\gamma}ghdz - \mathbf{i}\int_{\gamma}\overline{hdz}=0, \quad 
\int_{\gamma}hdz +\mathbf{i}\int_{\gamma} \overline{ghdz}=0 
$$
for any loop $\gamma\in H_{1}(M, \Z)$. Thus we get that 
$$
2\pi\mathbf{i}(d^{\, j}_{-1}+\mathbf{i}\, \overline{c^{\, j}_{-1}})=0, \quad 
2\pi\mathbf{i}(c^{\, j}_{-1}-\mathbf{i}\overline{\, d^{j}_{-1}})=0, 
$$
and therefore $c^{\, j}_{-1}=d^{\, j}_{-1}=0$. Hence ${\mu}_{j}=0$, contrary to our assumption. 
Applying the Riemann-Roch theorem to $hdz$ 
on $\overline{M}_{\gamma}$, we obtain that 
$$
d- \displaystyle \sum_{j=1}^{k}{\mu}_{j}= 2\gamma -2, 
$$
where $d$ denotes the degree of $g$ considered as a map of $\overline{M}_{\gamma}$. Thus we have 
\begin{equation}\label{eq:proof-Osserman}
d=2\gamma -2+ \displaystyle \sum_{j=1}^{k}{\mu}_{j}\geq 2\gamma -2 +2k > 2\gamma -2+k= 2(\gamma -1+ k/2), 
\end{equation}
and 
$$
\dfrac{1}{R}<\dfrac{1}{2}. 
$$

On the other hand, we assume that $g$ omits $D_{g}$ values. 
Let $n_{0}$ be the sum of the branching orders 
at the image of these exceptional values. Then we have 
$$
k \geq d D_{g} - n_{0}. 
$$
Let $n_{g}$ be the total branching order of $g$ on $\overline{M}_{\gamma}$. Then applying the 
Riemann-Hurwitz formula to the meromorphic function $g$ on $\overline{M}_{\gamma}$, we have 
\begin{equation}\label{eq:R-H}
n_{g}=2(d+\gamma -1). 
\end{equation}
Hence we have 
\begin{equation}\label{eq:omitted-R}
D_{g}\leq \dfrac{n_{0}+k}{d}\leq \dfrac{n_{g}+k}{d}=2+\dfrac{2}{R}. 
\end{equation}
\qed

\begin{remark}\label{rem-Osserman}
Since
$$
\displaystyle \int_{M} K_{ds^{2}}dA=-2\pi d, 
$$
the inequality in (\ref{eq:omitted-R})
\begin{equation}\label{eq:CO}
d \geq 2\gamma -2 +2k
\end{equation}
corresponds to the Chern-Osserman inequality (\cite{CO1967}, see also \cite{Um2010}) 
for complete minimal surfaces in ${\R}^{4}(={\C}^{2})$. 
Note that Kokubu, Umehara and Yamada \cite{KUY2002} proved that the equality of (\ref{eq:CO}) holds if and only if 
each end is asymptotic to a catenoid-type end or a planar end in some 3-dimensional subspace in ${\R}^{4}(={\C}^{2})$. 
In particular, all ends are embedded. Inequality (\ref{eq:CO}) is useful in showing the global properties of complete 
minimal Lagrangian surfaces with finite total curvature in ${\C}^{2}$. For instance, we can show that the following example 
is the unique complete minimal Lagrangian surface of the finite total curvature $-2\pi$ in ${\C}^{2}$. 

\begin{corollary}\label{cor:2pi}
A complete minimal Lagrangian surface in ${\C}^{2}$ whose total curvature is $-2\pi$ must have the following data: 
\begin{equation}\label{eq:Enneper}
F=(F_{1}, F_{2}) \colon \C\to {\C}^{2}, \quad (F_{1}, F_{2})=(az^{2}+b,\, 2az+c) \quad (a, b, c\in \C).
\end{equation}
\end{corollary}
\begin{proof}
Since $d=1$, $g\colon M\to S^{2}(=\C\cup \{\infty \})$ is a conformal diffeomorphism, and $M$ is conformally equivalent 
to $S^{2}\setminus \{p_{1}, \cdots , p_{k}\}$. However, by (\ref{eq:CO}), 
$2(k-1)\leq 1$. Thus $k=1$. If $k=1$, then $M$ is conformally equivalent to the complex plane $\C$. In fact, 
we may identify $M$ with $g(M)\subset \C\cup \{\infty \}$ and assume $(\C\cup \{\infty \}) \backslash g(M)=\{\infty \}$ 
after a suitable M\"obius transformation. Then $g(z)=z$ and the holomorphic $1$-form $hdz$ has no zeros on $\C$ because 
the metric $ds^{2}$ is nondegenerate. Since $hdz$ is extended meromorphically to $S^{2}=\C\cup \{\infty\}$, $h(z)$ is 
a polynomial in $z$ and therefore constant. Set that $h=2a\,(a\in \C)$. Then we have 
$$
F_{1}=\int gh\,dz=az^{2}+b, \qquad F_{2}=\int hdz=2az+c. 
$$
\end{proof}
\end{remark}

\subsection{Proof of Theorem \ref{great-Picard1}}\label{section3-2}
We first recall the notion of chordal distance between two distinct values in the Riemann sphere $\C\cup \{\infty \}$. 
For two distinct values $\alpha$, $\beta\in \C\cup \{\infty\}$, we set 
$$
|\alpha, \beta|:= \dfrac{|\alpha -\beta|}{\sqrt{1+|\alpha|^{2}}\sqrt{1+|\beta|^{2}}}
$$
if $\alpha \not= \infty$ and $\beta \not= \infty$, and $|\alpha, \infty|=|\infty, \alpha| := 1/\sqrt{1+|\alpha|^{2}}$. 
We note that, if we take $v_{1}$, $v_{2}\in {\Si}^{2}$ with $\alpha =\varpi (v_{1})$ and $\beta = \varpi (v_{2})$, we have that 
$|\alpha, \beta|$ is a half of the chordal distance between $v_{1}$ and $v_{2}$, where $\varpi$ denotes the stereographic projection of 
the $2$-sphere ${\Si}^{2}$ onto $\C\cup \{\infty \}$. 

Before proceeding to the proof of Theorem \ref{great-Picard1}, we recall two function-theoretic lemmas. 
\begin{lemma}{\cite[(8.12) in page 136]{Fu1997}}\label{Lemma2-1}
Let $g$ be a nonconstant meromorphic function on ${\Delta}_{R}=\{z\in \C ; |z|< R \}$ $(0<R\leq +\infty)$ which omits $q$ values 
${\alpha}_{1}, \cdots, {\alpha}_{q}$. If $q>2$, then for each positive $\eta$ with $\eta <(q-2)/q$, then there exists a positive constant $C'$, 
depending on $q$ and $L:=\min_{i< j}|{\alpha}_{i}, {\alpha}_{j}|$, such that 
\begin{equation}\label{Lemma2-1-1}
\dfrac{|g'_{z}|}{(1+|g|^{2})\prod_{j=1}^{q}|g, {\alpha}_{j}|^{1-\eta}}\leq C'\dfrac{R}{R^{2}-|z|^{2}}. 
\end{equation}  
\end{lemma}

\begin{lemma}{\cite[Lemma 1.6.7]{Fu1993}}\label{Lemma2-2} 
Let $d{\sigma}^{2}$ be a conformal flat metric on an open Riemann surface $\Sigma$. 
Then, for each point $p\in \Sigma$, there exists a local diffeomorphism $\Phi$ of a 
disk ${\Delta}_{R}=\{z\in \C ; |z|< R \}$ $(0<R\leq +\infty)$ onto an open 
neighborhood of $p$ with $\Phi (0)=p$ such that $\Phi$ is a local isometry, that is, 
the pull-back ${\Phi}^{\ast}(d{\sigma}^{2})$ is equal to the standard Euclidean metric $ds^{2}_{Euc}$ on ${\Delta}_{R}$ 
and, for a point $a_{0}$ with $|a_{0}|=1$, the ${\Phi}$-image ${\Gamma}_{a_{0}}$ of the curve $L_{a_{0}}=\{w:= a_{0}s ; 0 < s < R\}$ 
is divergent in $\Sigma$. 
\end{lemma}

\begin{proof}[{\it Proof of Theorem \ref{great-Picard1}}] 
Suppose that the Gauss map $g$ attains four distinct values ${\alpha}_{1}, \cdots, {\alpha}_{4}$ only a finite number of times. 
We may assume that ${\alpha}_{4}= \infty$ after a suitable M\"obius transformation. 
Then the assumption of the theorem implies 
that outside a compact subset $D$ in $M$, $g$ is holomorphic and omits three values ${\alpha}_{1}$, ${\alpha}_{2}$, ${\alpha}_{3}$. 
We choose a positive number $\eta$ with $0< \eta < 1/4$ and set $\lambda := 1/(2-4\eta)$. Now we define a new metric 
\begin{equation}\label{eq1:Thm1-4}
\displaystyle d{\sigma}^{2}=|h|^{\frac{2}{1-\lambda}} \biggl( \frac{1}{|g'_{z}|} 
{\prod}_{j=1}^{3}\biggl( \dfrac{|g -{\alpha}_{j}|}{\sqrt{1+|{\alpha}_{j}|^2}}{\biggr)}^{1-\eta} 
\biggr)^{\frac{2\lambda}{1-\lambda}}|dz|^{2} 
\end{equation}
on the set $M':=\{p\in M\backslash D \, ; \, g'_{z}(p)\not=0\}$. 
Since $h$ and $g$ are holomorphic, $d{\sigma}^{2}$ is flat and can be smoothly extended over $D$. 
We thus obtain the metric $d{\sigma}^{2}$ on $M'':= M'\cup D$ which is flat outside the compact set $D$. 

We prove that $d{\sigma}^{2}$ is complete on $M''$. This will be proved by contradiction. 
If $d{\sigma}^{2}$ is not complete on $M''$, then there exists a divergent curve $\gamma (t)$ on $M''$ 
with finite length. By removing an initial segment, if necessary, we may assume that there exists a positive 
distance $d$ between the curve $\gamma$ and the compact set $D$. Thus $\gamma \colon [0, 1)\to M'$ and, since 
$\gamma$ is divergent on $M''$ with finite length, the curve $\gamma (t)$ tends to either a point where $g'_{z}=0$ 
or else the boundary of $M$ as $t\to 1$. However, the former case cannot occur. The reason is as follows. We assume that 
$\gamma (t)$ tends to a point $p_{0}$ where $g'_{z}=0$ as $t\to 1$. Taking a local complex coordinate $\zeta := g'_{z}$ 
in a neighborhood of $p_{0}$ with $\zeta (p_{0})=0$, we can write 
$$
d{\sigma}^{2}=|\zeta|^{-2\lambda /(1-\lambda)} w|d\zeta|^{2}
$$
for some positive smooth function $w$. Since $\lambda /(1-\lambda)> 1$, we have 
$$
\int_{\gamma} d\sigma \geq \widetilde{C}\int_{\gamma}\dfrac{|d\zeta|}{|\zeta|^{\lambda /(1-\lambda)}}= +\infty, 
$$
where $\widetilde{C}$ is some positive number. It contradicts that $\gamma$ has finite length. 
We conclude that $\gamma (t)$ must tend to the boundary of $M$ when $t\to 1$. 

Then we choose a number $t_{0}\in [0, 1)$ such that the length of $\gamma ([t_{0}, 1))$ with respect to the metric $d{\sigma}^{2}$ 
is less than $d/3$. Since $d{\sigma}^{2}$ is flat, by Lemma \ref{Lemma2-2}, there exists an isometry $\Phi$ satisfying 
$\Phi (0)=\gamma (t_{0})$ from a disk $\triangle_{R}=\{w\in\C \,;\, |w|<R \}$ ($0< R \leq +\infty$) with the standard Euclidean 
metric $ds^{2}_{\bf E}$ onto an open neighborhood of the point $\gamma (t_{0})$ with the metric $d{\sigma}^{2}$, such that, for a 
point $a_{0}$ with $|a_{0}|=1$, the $\Phi$-image $\Gamma_{a_{0}}$ of the line segment $L_{a_{0}}:=\{w:=a_{0}s\,;\, 0<s<R \}$ 
is divergent in $M'$. Since the length of $\gamma ([t_{0}, 1))$ is less than $d/3$ and $\gamma$ is divergent curve in $M$, 
we have $R\leq d/3$. Hence the image $\Phi (\triangle_{R})$ must be bounded away from $D$ by a distance of at least $2d/3$. 
Moreover, since the zeros of $g'_z$ have been shown to be infinitely far away in the metric, $\Gamma_{a_{0}}$ must actually 
go to the boundary of $M$. 

For brevity, we denote the function $g\circ \Phi$ on $\triangle_{R}$ by $g$ in the following. 
Since ${\Phi}^{\ast}d{\sigma}^{2}=|dz|^{2}$, we get by (\ref{eq1:Thm1-4}) that 
\begin{equation}\label{eq2:Thm1-4}
\displaystyle |h|=\biggl(|{g}'_{z}|\prod_{j=1}^{3}
\biggl(\frac{\sqrt{1+|{\alpha}_{j}|^{2}}}{|g -{\alpha}_{j}|} \biggr)^{1-\eta} \biggr)^{\lambda}\,.
\end{equation}
By Lemma \ref{Lemma2-1}, we have 
\begin{eqnarray*}
{\Phi}^{\ast}ds    &=& |h|\sqrt{1+|g|^{2}}|dz| \\
                   &=& \displaystyle \biggl( |{g}'_{z}|(1+|g|^{2})^{1/2\lambda} \prod_{j=1}^{3}
                          \biggl( \frac{\sqrt{1+|{\alpha}_{j}|^{2}}}{|g -{\alpha}_{j}|}\biggr)^{1-\eta}\biggr)^{\lambda} |dz| \\
                   &=& \displaystyle \biggl( \dfrac{|{g}'_{z}|}{(1+|g|^{2})\prod_{j=1}^{4} |g, {\alpha}_{j}|^{1-\eta}} \biggr)^{\lambda}|dz|  \\
                   &\leq& (C')^{\lambda}\biggl(\dfrac{R}{R^{2}-|z|^{2}} \biggr)^{\lambda}|dz|. 
\end{eqnarray*}
Then, for the length $L$ of ${\Gamma}_{a_{0}}$ with respect to the metric $ds^{2}$, we obtain 
$$
L= \int_{{\Gamma}_{a_{0}}} ds=\int_{L_{a_{0}}} {\Phi}^{\ast}ds\leq (C')^{\lambda}\int_{L_{a_{0}}} 
\biggl{(} \frac{R}{R^{2}-|z|^{2}}\biggr{)}^{\lambda}|dz|\leq (C')^{\lambda}\dfrac{R^{1-\lambda}}{1-\lambda} <+\infty 
$$
because $1/2< \lambda < 1$. In particular, the length of $L$ is finite. 
We thus show that if $d{\sigma}^{2}$ were not complete on $M''$, then we could find a divergent curve on $M$ 
with finite length in the original metric $ds^{2}$, that is, $M$ would not be complete. 
However it contradicts the assumption. Hence we conclude that $d{\sigma}^{2}$ is complete on $M''$. 

Since $d{\sigma}^{2}$ is also flat outside a compact set, the total curvature of $d{\sigma}^{2}$ is finite. 
Then, by the Huber theorem \cite{Hu1957}, $M''$ is finitely connected. We thus show that $g'_{z}$ 
can have only a finite number of zeros and $M$ is finitely connected. Moreover, by \cite[Theorem 2.1]{Os1963}, 
each annular end of $M''$, hence of $M$, is conformally equivalent to a punctured disk. Therefore the Riemann surface $M$ 
must be conformally equivalent to $\overline{M}\backslash \{p_{1}, \cdots, p_{k}\}$, where 
$\overline{M}$ is a closed Riemann surface and $p_{j}\in \overline{M}$ $(j=1, \ldots, k). $ 
In a neighborhood of each of $p_{j}$, $g$ is holomorphic 
and omits three values. By the Picard great theorem, $g$ cannot have an essential singularity, but must have at most a pole. 
Hence $g$ extends to a meromorphic function on $\overline{M}$. 
Then we have 
$$
\displaystyle \int_{M} K_{ds^{2}}\, dA= -\dfrac{1}{2} \int_{M} \biggl{(}\dfrac{2|g'_{z}|}{1+|g|^{2}} \biggr{)}^{2}du\wedge dv=-2\pi d, 
$$
where $z=u+{\bf i}v$ and 
$d$ denotes the degree of $g$ considered as a map on $\overline{M}$. 
In particular, $M$ must have finite total curvature. 
\end{proof} 


\end{document}